\documentclass[12pt]{article}
\usepackage{amssymb}
\usepackage{amsthm}
\usepackage{amsmath}
\usepackage{graphicx}
\usepackage{enumerate}

\DeclareMathOperator{\Max}{Max}

\DeclareRobustCommand*{\tautequiv}{\Relbar\joinrel\mathrel|\mathrel|\joinrel\Relbar}

\newtheorem{theorem}{Theorem}[section]
\newtheorem{definition}[theorem]{Definition}
\newtheorem{lemma}[theorem]{Lemma}
\newtheorem{proposition}[theorem]{Proposition}

\title{The logic of orthomodular posets of finite height}
\author{Ivan~Chajda and Helmut~L\"anger}
\date{}
\begin{document}
\footnotetext[1]{Support of the research by the Austrian Science Fund (FWF), project I~4579-N, and the Czech Science Foundation (GA\v CR), project 20-09869L, entitled ``The many facets of orthomodularity'', as well as by \"OAD, project CZ~02/2019, and, concerning the first author, by IGA, project P\v rF~2020~014, is gratefully acknowledged.}
\maketitle
\begin{abstract}
Orthomodular posets form an algebraic formalization of the logic of quantum mechanics. The question is how to introduce the connective implication in such a logic. We show that this is possible when the orthomodular poset in question is of finite height. The main point is that the corresponding algebra, called implication orthomodular poset, i.e.\ a poset equipped with a binary operator of implication, corresponds to the original orthomodular poset and this operator is everywhere defined. We present here the complete list of axioms for implication orthomodular posets. This enables us to derive an axiomatization in Gentzen style for the algebraizable logic of orthomodular posets of finite height.
\end{abstract}



\section{Introduction}

Any physical theory determines a class $(\mathcal E,S)$ of event-state systems where $\mathcal E$ contains the events that may occur in this system and $S$ contains the states that such a physical system may assume. In quantum mechanics $\mathcal E$ is usually identified with the set $\mathcal P(\mathbf H)$ of all projection operators of a Hilbert space $\mathbf H$. This set is in one-to-one correspondence with the set of all closed subspaces of $\mathbf H$. Hence $\mathcal P(\mathbf H)$ forms an orthomodular lattice when endowed with intersection and with the following operations defined for all $P,Q\in\mathcal P(\mathbf H)$:
\begin{itemize}
\item $P':=I-P$, where $I$ is the identity operator,
\item $P\vee Q$ is the projection onto the closed subspace generated by the union of the closed subspaces associated with $P$ and $Q$, respectively.
\end{itemize}
The relevance of such an approach is restricted by the fact that $P\vee Q$ need not be defined for all $P,Q\in\mathcal P(\mathbf H)$. However, if $P,Q$ are orthogonal, for short $P\perp Q$, which means $P\leq Q'$, then $P\vee Q$ exists. Hence, instead of orthomodular lattices frequently orthomodular posets are used as an algebraic axiomatization of the logic of quantum mechanics.

From the formal algebraic point of view, an algebraic system can be considered as a logic if it contains a connective implication which enables to derive propositions by means of axioms and rules. The problem of how to introduce a connective implication in an orthomodular poset was successfully solved by the authors in \cite{CLb}. The present paper is devoted to the following two topics:
\begin{itemize}
\item Introduce an assigned implication algebra derived from a given orthomodular poset and determine the correspondence between such an algebra and the original orthomodular poset,
\item establish a Gentzen style axiomatic system for this implication algebra, i.e.\ determine a system of axioms and rules corresponding to the assigned implication algebra.
\end{itemize}
It is worth noticing that for orthomodular lattices the mentioned tasks were solved by the first author and J.~Cirulis in \cite{CC} where implication was sectional orthocomplementation, the so-called Dishkant implication. For orthomodular posets such an approach is not possible since the operation join is not everywhere defined. Hence we restrict ourselves to orthomodular posets of finite height where every non-empty subset has maximal elements. Then implication can be defined not as an operation, but as an operator taking maximal elements of a certain lower cone. Hence the result of implication need not be a single element, but a set of elements formed by using certain maximal elements.

We start with several definitions of concepts concerning posets with a unary operation.

A {\em bounded poset} is an ordered quadruple $\mathbf P=(P,\leq,0,1)$ such that $(P,\leq)$ is a poset, $0,1\in P$ and $0\leq x\leq1$ for all $x\in P$. Let $\mathbf P=(P,\leq,0,1)$ be a bounded poset and $a,b\in P$ and $A,B\subseteq P$. Then $\mathbf P$ is said to be of {\em finite height} if every chain is finite. A unary operation $'$ on $P$ is called {\em antitone} if for $x,y\in P$, $x\leq y$ implies $y'\leq x'$, and it is called an {\em involution} if it satisfies the identity $x''\approx x$. We define
\begin{align*}
L(A) & :=\{x\in P\mid x\leq y\text{ for all }y\in A\}, \\
U(A) & :=\{x\in P\mid y\leq x\text{ for all }y\in A\}.
\end{align*}
Instead of $L(\{a\})$, $L(\{a,b\})$, $L(A\cup B)$ we shortly write $L(a)$, $L(a,b)$, $L(A,B)$. Analogously we proceed in similar cases. It is clear that if $\mathbf P$ is of finite height then every element of $A$ lies below a maximal element of $A$. Let $'$ be an antitone involution on $\mathbf P$. We say that $a$ and $b$ are {\em orthogonal} to each other, shortly $a\perp b$, if $a\leq b'$. By $\Max A$ we denote the set of all maximal elements of $A$. Moreover, by $a\vee A$ we mean the set $\{a\vee x\mid x\in A\}$. By $A\leq a$ we mean $x\leq a$ for all $x\in A$. We put $A':=\{x'\mid x\in A\}$. Finally, De Morgan's laws hold, i.e.\, if $a\vee b$ is defined then so is $a'\wedge b'$ and we have $(a\vee b)'=a'\wedge b'$. Dually, if $a\wedge b$ is defined then so is $a'\vee b'$ and we have $(a\wedge b)'=a'\vee b'$. 

\begin{definition}\label{def1}
An {\em orthomodular poset} is an ordered quintuple $(P,\leq,{}',0,1)$ such that $(P,\leq,0,1)$ is a bounded poset, $'$ is an antitone involution which is a complementation, i.e.\ which satisfies the identities $x\vee x'\approx1$ and $x\wedge x'\approx0$, where in case $x\perp y$ the supremum $x\vee y$ is defined and which satisfies the {\em orthomodular law}
\begin{enumerate}
\item[{\rm(OM)}] $x\leq y$ implies $(y'\vee x)'\vee x=y$.
\end{enumerate}
\end{definition}

The expression in (OM) is well defined. This can be seen as follows: If $x\leq y$ then $y'\leq x'$ and hence $y'\perp x$ which shows that $y'\vee x$ exists. Now $x\leq y'\vee x$ and hence $(y'\vee x)'\leq x'$, i.e.\ $(y'\vee x)'\perp x$ which shows that $(y'\vee x)'\vee x$ exists. Due to De Morgan's laws, (OM) can be written in the form
\[
x\leq y\text{ implies }x\vee(y\wedge x')=y.
\]

\section{Implication in orthomodular posets}

In an orthomodular poset $\mathbf P=(P,\leq,{}',0,1)$ the operations of join and meet are only partial. Since orthomodular posets are considered as logics of quantum mechanics, we need to have the logical connective implication in order to enable deducing propositions in this logic. For this, the implication should be an everywhere defined binary operator on $\mathbf P$. For orthomodular lattices one usually considers the Sasaki implication, see e.g.\ \cite{Be} or \cite{CC}. Here we show that that if $\mathbf P$ is of a finite height then an everywhere defined implication can be introduced also here but it will not be a binary operation on $P$ since the result of the implication need not be a single element, but may be a subset of $P$. On the other hand, this subset will be as small as possible.

\begin{definition}\label{def3}
Let $\mathbf P=(P,\leq,{}',0,1)$ be an orthomodular poset of finite height. We define
\[
x\rightarrow y:=y\vee\Max L(x',y')
\]
for all $x,y\in P$.
\end{definition}

Since $\Max L(x',y')\leq y'$ for all $x,y\in P$, $x\rightarrow y$ is defined for all $x,y\in P$. The expression $x\rightarrow y=1$ can be extended to the case where $x$ and $y$ are substituted by subsets $A$ and $B$ of $P$, respectively, as follows:
\begin{enumerate}
\item[(G)] $A\rightarrow B=1$ if and only if for every $x\in A$ there exists some $y\in B$ with $x\rightarrow y=1$.
\end{enumerate}
In the following we often identify singletons with the unique element they contain.

At first we list several simple but essential properties of this implication.

\begin{proposition}\label{prop1}
Let $(P,\leq,{}',0,1)$ be an orthomodular poset of finite height and $\rightarrow$ be defined as in Definition~\ref{def3}. Then we have the following for all $x,y\in P$:
\begin{enumerate}[{\rm(i)}]
\item $x\rightarrow0\approx x'$,
\item $x\leq y$ is equivalent to $x\rightarrow y=1$,
\item $x\rightarrow x'\approx x'$,
\item if $x\perp y$ then $(x\vee y)'\vee y$ exists and $x\rightarrow y=(x\vee y)'\vee y$,
\item if $x\perp y$ then $x\vee y=(x\rightarrow y)\rightarrow y$,
\item if $x\perp y$ then $x\vee y=x'\rightarrow y$,
\item if $x\geq y$ then $x\rightarrow y=x'\vee y$,
\item if $x\perp y$ then $((x\rightarrow y)\rightarrow y)\rightarrow y=x\rightarrow y$,
\item if $x\leq y$ then $(((y'\rightarrow x)\rightarrow x)'\rightarrow x)\rightarrow x=y$,
\item $(x'\rightarrow x)\rightarrow x\approx1$,
\item $x\rightarrow(y\rightarrow x)\approx1$.
\end{enumerate}
\end{proposition}

\begin{proof}
\
\begin{enumerate}[(i)]
\item $x\rightarrow0\approx0\vee\Max L(x',1)\approx0\vee x'\approx x'$
\item If $x\leq y$ then $x\rightarrow y=y\vee\Max L(x',y')=y\vee\Max L(y')=y\vee y'=1$. Conversely, if $x\rightarrow y=1$ and $z\in\Max L(x',y')$ then because of $y\vee\Max L(x',y')=1$ we have $y\vee z=1$ and, since $z\leq y'$, we conclude
\[
y'=(y\vee z)'\vee z=z\in\Max L(x',y')\subseteq L(x',y')\subseteq L(x')
\]
whence $y'\leq x'$, i.e.\ $x\leq y$.
\item $x\rightarrow x'\approx x'\vee\Max L(x',x)\approx x'\vee0\approx x'$
\item If $x\perp y$ then $x\vee y$ is defined and hence so is $x'\wedge y'$ whence
\[
x\rightarrow y=y\vee\Max L(x',y')=y\vee\Max L(x'\wedge y')=y\vee(x'\wedge y')=(x\vee y)'\vee y.
\]
\item If $x\perp y$ then, using (iv), we get
\begin{align*}
x\vee y & =((x\vee y)'\vee y)'\vee y=y\vee((x\vee y)\wedge y')=y\vee\Max L((x\vee y)\wedge y')= \\
        & =y\vee\Max L((x\vee y)\wedge y',y')=y\vee\Max L((x\rightarrow y)',y')=(x\rightarrow y)\rightarrow y.
\end{align*}
\item If $x\perp y$ then $x\vee y=y\vee x=y\vee\Max L(x)=y\vee\Max L(x,y')=x'\rightarrow y$.
\item If $x\geq y$ then $x'\perp y$ and hence $x\rightarrow y=x''\rightarrow y=x'\vee y$ by (vi).
\item If $x\perp y$ then, by (v), (vii) and (iv),
\[
((x\rightarrow y)\rightarrow y)\rightarrow y=(x\vee y)\rightarrow y=(x\vee y)'\vee y=x\rightarrow y.
\]
\item If $x\leq y$ then
\[
(((y'\rightarrow x)\rightarrow x)'\rightarrow x)\rightarrow x=(y'\vee x)'\vee x=y
\]
according to (v) and orthomodularity.
\item According to (v), $(x'\rightarrow x)\rightarrow x\approx x'\vee x\approx1$.
\item This follows from $x\leq y\rightarrow x$.
\end{enumerate}
\end{proof}

Now we introduce our main concept.

\begin{definition}\label{def2}
An {\em implication orthomodular poset} is an ordered triple $(I,\rightarrow,0)$ such that $\rightarrow$ is a mapping from $I^2$ to $2^I\setminus\{\emptyset\}$ and $0\in I$ satisfying the following conditions for $x,y,z\in I$ {\rm(}$x'$ is an abbreviation for $x\rightarrow0$ and $1$ an abbreviation for $0'$ and we use convention {\rm(G))}:
\begin{enumerate}[{\rm(O1)}]
\item $0\rightarrow x\approx x\rightarrow x\approx1$,
\item if $x\rightarrow y=y\rightarrow x=1$ then $x=y$,
\item if $x\rightarrow y=y\rightarrow z=1$ then $x\rightarrow z=1$,
\item $x''\approx x$,
\item if $x\rightarrow y=1$ then $(y\rightarrow z)\rightarrow(x\rightarrow z)=1$,
\item if $x\rightarrow y=1$ then $(((y'\rightarrow x)\rightarrow x)'\rightarrow x)\rightarrow x=y$,
\item if $x\rightarrow y'=1$ then $x\rightarrow((x\rightarrow y)\rightarrow y)=y\rightarrow((x\rightarrow y)\rightarrow y)=1$,
\item if $x\rightarrow y'=x\rightarrow z=y\rightarrow z=1$ then $((x\rightarrow y)\rightarrow y)\rightarrow z=1$,
\item $(x'\rightarrow x)\rightarrow x\approx1$,
\item $x\rightarrow(y\rightarrow x)\approx1$.
\end{enumerate}
\end{definition}

As mentioned above, the results of this operator $\rightarrow$ need not be singletons, but may be subsets containing more than one element. However, we show that $a\rightarrow0$ is a singleton for all $a\in I$.

\begin{lemma}\label{lem2}
Let $(I,\rightarrow,0)$ be an implication orthomodular poset and $a\in I$. Then $a'$ is a singleton.
\end{lemma}

\begin{proof}
If $b\in a'$ then $b'\subseteq a''=a$ because of (G) and (O4) whence $b'=a$ which implies $a'=b''=b$ again by (O4). This shows that $a'$ is a singleton.
\end{proof}

Let us note that the name {\em implication orthomodular poset} is not misleading since if one defines $x\leq y$ if $x\rightarrow y=1$, $x':=x\rightarrow0$ then $(I,\leq,{}',0,0')$ is in fact an orthomodular poset (see Theorem~\ref{th2} below).

We say that an implication orthomodular poset $(I,\rightarrow,0)$ is of {\em finite height} if there does not exist an infinite sequence $a_1,a_2,a_3,\ldots$ of pairwise different elements of $I$ satisfying $a_n\rightarrow a_{n+1}=1$ for all $n\geq1$.

The next theorem shows that implication orthomodular posets arise in a natural and expected way from orthomodular posets.

\begin{theorem}\label{th1}
Let $\mathbf P=(P,\leq,{}',0,1)$ be an orthomodular poset of finite height and put
\[
x\rightarrow y:=y\vee\Max L(x',y')
\]
for all $x,y\in P$. Then $\mathbb I(\mathbf P):=(P,\rightarrow,0)$ is an implication orthomodular poset of finite height.
\end{theorem}

\begin{proof}
Throughout the proof we use (i), (ii), (v) and (ix) of Proposition~\ref{prop1} and $0'=1$.
\begin{enumerate}
\item[(O1)] -- (O3) follow from the fact that $(P,\leq,0,1)$ is a bounded poset.
\item[(O4)] follows from the fact that $'$ is an involution.
\item[(O5)] Assume $x\rightarrow y=1$. Then, by (ii) of Proposition~\ref{prop1}, $x\leq y$. Let $u\in y\rightarrow z$. Then there exists some $w\in\Max L(y',z')$ with $u=z\vee w$. Since $y'\leq x'$, we have $w\in L(x',z')$. Hence there exists some $t\in\Max L(x',z')$ with $w\leq t$. Now $u=z\vee w\leq z\vee t\in x\rightarrow z$. This shows $(y\rightarrow z)\rightarrow(x\rightarrow z)=1$.
\item[(O6)] If $x\rightarrow y=1$ then, by (ii) of Proposition~\ref{prop1}, $x\leq y$ and hence by (ix) of Proposition~\ref{prop1}, $(((y'\rightarrow x)\rightarrow x)'\rightarrow x)\rightarrow x=y$.
\item[(O7)] If $x\rightarrow y'=1$ then, by (ii) of Proposition~\ref{prop1}, $x\leq y'$, i.e.\ $x\perp y$ and hence $(x\rightarrow y)\rightarrow y=x\vee y$ by (v) whence $x,y\leq(x\rightarrow y)\rightarrow y$, i.e.\ $x\rightarrow((x\rightarrow y)\rightarrow y)=y\rightarrow((x\rightarrow y)\rightarrow y)=1$.
\item[(O8)] If $x\rightarrow y'=x\rightarrow z=y\rightarrow z=1$ then $x,y\leq z$ and $x\leq y'$ by (ii) of Proposition~\ref{prop1}, i.e.\ $x\perp y$, and hence $(x\rightarrow y)\rightarrow y=x\vee y$ by (v) of Proposition~\ref{prop1} whence $(x\rightarrow y)\rightarrow y\leq z$, i.e.\ $((x\rightarrow y)\rightarrow y)\rightarrow z=1$.
\item[(O9)] Since $x'\leq x'$ we have $x'\perp x$ and, by (v) of Proposition~\ref{prop1}, $(x'\rightarrow x)\rightarrow x\approx x'\vee x\approx1$.
\item[(O10)] This follows from $y\rightarrow x\geq x$.
\end{enumerate}
Of course, $\mathbb I(\mathbf P)$ is of finite height.
\end{proof}

If $\mathbf P$ is an orthomodular poset of a finite height then $\mathbb I(\mathbf P)$ will be referred to as {\em the implication orthomodular poset assigned to $\mathbf P$}.

That in fact an implication orthomodular poset can be considered as a poset is shown by the next theorem.

\begin{theorem}\label{th2}
Let $\mathbf I=(I,\rightarrow,0)$ be an implication orthomodular poset of finite height and put
\begin{align*}
x\leq y & :\Leftrightarrow x\rightarrow y=1, \\
     x' & :=x\rightarrow0, \\
      1 & :=0'
\end{align*}
{\rm(}$x,y\in I${\rm)}. Then $\mathbb P(\mathbf I):=(I,\leq,{}',0,1)$ is an orthomodular poset of finite height.
\end{theorem}

\begin{proof}
$(I,\leq,0,1)$ is a bounded poset because of (O1) -- (O3) and (O10), $'$ is a unary operation on $I$ because of Lemma~\ref{lem2}, $'$ is antitone because of (O5) and $'$ is an involution because of (O4). ($x\leq1$ follows from $x\rightarrow1\approx x\rightarrow(0\rightarrow x)\approx1$ by (O1) and (O10).) Assume $x\perp y$. Then (O7) says that $(x\rightarrow y)\rightarrow y$ is an upper bound of $x$ and $y$, and (O8) says that every upper bound of $x$ and $y$ is greater than or equal to $(x\rightarrow y)\rightarrow y$. Together, we obtain (v) of Proposition~\ref{prop1}. Since $x'\leq x'$ we have $x'\perp x$ which implies $x'\vee x\approx(x'\rightarrow x)\rightarrow x\approx1$ by (v) of Proposition~\ref{prop1} and (O9). Since $'$ is an antitone involution on the bounded poset $(I,\leq,0,1)$ we conclude $x\wedge x'\approx x''\wedge x'\approx(x'\vee x)'\approx1'\approx0$ by De Morgan's laws. This shows that $'$ is a complementation. Finally, if $x\leq y$ then $x\rightarrow y=1$ and because of (v) of Proposition~\ref{prop1} we have $(y'\rightarrow x)\rightarrow x=y'\vee x$ and again by (v) of Proposition~\ref{prop1} and (O6) we conclude
\[
(y'\vee x)'\vee x=((y'\vee x)'\rightarrow x)\rightarrow x=(((y'\rightarrow x)\rightarrow x)'\rightarrow x)\rightarrow x=y
\]
proving orthomodularity. Of course $\mathbb P(\mathbf I)$ is of finite height.
\end{proof}

If $\mathbf I$ is an implication orthomodular poset of finite height then $\mathbb P(\mathbf I)$ will be referred to as {\em the orthomodular poset assigned to $\mathbf I$}.

Consider the following condition for implication orthomodular posets $(I,\rightarrow,0)$:
\begin{enumerate}
\item[(C)] $x\rightarrow y=\{(y\rightarrow u)\rightarrow u\mid u\in I, u\rightarrow x'=u\rightarrow y'=1, u=w\text{ for all }w\in I\text{ with }u\rightarrow w=w\rightarrow x'=w\rightarrow y'=1\}$
\end{enumerate}
($x,y\in I$). If $\mathbf P$ is an orthomodular poset of finite height then $\mathbb I(\mathbf P)=(P,\rightarrow,0)$ satisfies (C) because of the definition of $\rightarrow$, in fact it means
\[
x\rightarrow y=y\vee\Max L(x',y')
\]
for all $x,y\in P$.

To an orthomodular poset $\mathbf P$ of finite we may assign the corresponding implication orthomodular poset $\mathbb I(\mathbf P)$ and to this the corresponding orthomodular poset $\mathbb P(\mathbb I(\mathbf P))$. The question if $\mathbb P(\mathbb I(\mathbf P))=\mathbf P$ is answered in the following theorem. On the contrary, when starting with an implication orthomodular poset $\mathbf I$ of finite height, it turns out that $\mathbb I(\mathbb P(\mathbf I))=\mathbf I$ holds only in a special case.

\begin{theorem}
We have
\begin{enumerate}[{\rm(i)}]
\item $\mathbb P(\mathbb I(\mathbf P))=\mathbf P$ for every orthomodular poset $\mathbf P$ of finite height,
\item $\mathbb I(\mathbb P(\mathbf I))=\mathbf I$ for every implication orthomodular poset $\mathbf I$ of finite height satisfying {\rm(C)}.
\end{enumerate}
\end{theorem}

\begin{proof}
\
\begin{enumerate}[(i)]
\item Let
\begin{align*}
                      \mathbf P & =(P,\leq,{}',0,1)\text{ be an orthomodular poset of finite height}, \\
           \mathbb I(\mathbf P) & =(P,\rightarrow,0)\text{ its assigned implication orthomodular poset}, \\
\mathbb P(\mathbb I(\mathbf P)) & =(P,\sqsubseteq,^*,0,\bar1)
\end{align*}
and $a,b\in P$. Then $a^*=a\rightarrow0=a'$ because of (i), $\bar1=0^*=0'=1$ and the following are equivalent: $a\sqsubseteq b$; $a\rightarrow b=\bar1$; $a\rightarrow b=1$; $a\leq b$. This shows $\mathbb P(\mathbb I(\mathbf P))=\mathbf P$.
\item Let
\begin{align*}
                      \mathbf I & =(I,\rightarrow,0)\text{ be an implication orthomodular poset of finite height satisfying} \\
											          & \hspace*{5.5 mm}\text{(C)}, \\
           \mathbb P(\mathbf I) & =(I,\leq,{}',0,1)\text{ its assigned orthomodular poset}, \\
\mathbb I(\mathbb P(\mathbf I)) & =(I,\leadsto,0)
\end{align*}
and $c,d\in I$. Then $c\leadsto d=d\vee\Max L(c',d')=c\rightarrow d$ because of (C) and (v) of Proposition~\ref{prop1}.
\end{enumerate}
\end{proof}

\section{Deductive axiom system}

One way how to capture the logic of orthomodular posets is to construct an appropriate system of axioms and derivation rules, i.e.\ the so-called deductive system.

In the following we establish a so-called Gentzen axiom system for the logic of orthomodular posets of finite height, i.e.\ we algebraically axiomatize these posets by means of implication orthomodular posets. For this purpose, we recall some concepts taken from \cite{BP}.

For a class $\mathcal K$ of $\mathcal L$-algebras over a language $\mathcal L$, consider the relation $\models_\mathcal K$ that holds between a set $\Sigma$ of identities and a single identity $\varphi\approx\psi$ if every interpretation of $\varphi\approx\psi$ in a member of $\mathcal K$ holds provided each identity in $\Sigma$ holds under the same interpretation. In this case we say that $\varphi\approx\psi$ is a {\em $\mathcal K$-consequence} of $\Sigma$. The relation $\models_\mathcal K$ is called the {\em semantic equational consequence relation} determined by $\mathcal K$.

Given a deductive system $(\mathcal L,\vdash_L)$ over a language $\mathcal L$ with $\mathcal Fm$ denoting the class of its formulas, a class $\mathcal K$ of $\mathcal L$-algebras is called an {\em algebraic semantics} for $(\mathcal L,\vdash_L)$ if $\vdash_L$ can be interpreted in $\models_\mathcal K$ in the following sense: There exists a finite system $\delta_i(p)\approx\varepsilon_i(p)$, ($\delta(\varphi)\approx\varepsilon(\varphi)$, in brief) of identities with a single variable $p$ such that for all $\Gamma\cup\{\varphi\}\subseteq\mathcal Fm$,
\[
\Gamma\vdash_L\varphi\Leftrightarrow\{\delta(\varphi)\approx\varepsilon(\varphi),\psi\in\Gamma\}\models_\mathcal K\delta(\varphi)\approx\varepsilon(\varphi).
\]
Then $\delta_i\approx\varepsilon_i$ are called {\em defining identities} for $(\mathcal L,\vdash_L)$ and $\mathcal K$.

$\mathcal K$ is said to be {\em equivalent} to $(\mathcal L,\vdash_L)$ if there exists a finite system $\Delta_j(p,q)$ of formulas with two variables $p,q$ such that for every identity $\varphi\approx\psi$,
\[
\varphi\approx\psi\tautequiv_\mathcal K\delta(\varphi\Delta\psi)\approx\varepsilon(\varphi\Delta\psi),
\]
where $\varphi\Delta\psi$ means just $\Delta(\varphi,\psi)$ and $\Gamma\tautequiv_\mathcal K\Delta$ is an abbreviation for the conjunction $\Gamma\models_\mathcal K\Delta$ and $\Delta\models_\mathcal K\Gamma$. 

According to \cite R and \cite{RS}, a {\em standard system of implicative extensional propositional calculus} (SIC, for short) is a deductive system $(\mathcal L,\vdash_L)$ satisfying the following conditions:
\begin{itemize}
\item The language $\mathcal L$ contains a finite number of connectives of rank $0$, $1$ and $  2$ and none of higher rank,
\item $\mathcal L$ contains a binary connective $\rightarrow$ for which the following theorems and derived inference rules hold:
\begin{align*}
& \vdash\varphi\rightarrow\varphi, \\
& \varphi,\varphi\rightarrow\psi\vdash\psi, \\
& \varphi\rightarrow\psi,\psi\rightarrow\chi\vdash\varphi\rightarrow\chi, \\
& \varphi\rightarrow\psi,\psi\rightarrow\varphi\vdash P(\varphi)\rightarrow P(\psi)\text{ for every unary }P\in\mathcal L, \\
& \varphi\rightarrow\psi,\psi\rightarrow\varphi,\chi\rightarrow\lambda,\lambda\rightarrow\chi\vdash Q(\varphi,\chi)\rightarrow Q(\psi,\lambda)\text{ for every binary }Q\in\mathcal L.
\end{align*}
\end{itemize}
Since orthomodular posets are only partial algebras, it could be a problem to find appropriate algebraic semantics formulated by means of these partial operations. However, we have shown in Theorems~\ref{th1} and \ref{th2} that an orthomodular poset can equivalently be expressed as an implication orthomodular poset having only one binary operator $\rightarrow$ (cf.\ the symbol $Q$ used previously).

By the propositional logic $L_{{\rm OMP}}$ in a language $\mathcal L=\{\rightarrow,0\}$ ($\rightarrow$ is of rank 2, $0$ is of rank $0$) we understand a consequence relation $\vdash_{{\rm OMP}}$ (or $\vdash$, in brief) satisfying the axioms
\begin{enumerate}[(B1)]
\item $\vdash\varphi\rightarrow(\psi\rightarrow\varphi)$,
\item $\vdash\varphi\rightarrow\varphi$,
\item $\vdash\neg\neg\varphi\approx\varphi$,
\item $\vdash0\rightarrow\varphi$,
\item $\vdash(\neg\varphi\rightarrow\varphi)\rightarrow\varphi$
\end{enumerate}
and the rules
\begin{enumerate}
\item[(MP)] $\varphi,\varphi\rightarrow\psi\vdash\psi$,
\item[(Sf)] $\varphi\rightarrow\psi\vdash(\psi\rightarrow\chi)\rightarrow(\varphi\rightarrow\chi)$,
\item[(R1)] $\varphi\rightarrow\psi,\psi\rightarrow\varphi\vdash\varphi\approx\psi$,
\item[(R2)] $\varphi\rightarrow\psi\vdash(\neg((\neg\psi\rightarrow\varphi)\rightarrow\varphi)\rightarrow\varphi)\rightarrow\varphi\approx\psi$,
\item[(R3)] $\varphi\rightarrow\neg\psi\vdash\varphi\rightarrow((\varphi\rightarrow\psi)\rightarrow\psi)$,
\item[(R4)] $\varphi\rightarrow\neg\psi\vdash\psi\rightarrow((\varphi\rightarrow\psi)\rightarrow\psi)$,
\item[(R5)] $\varphi\rightarrow\neg\psi,\varphi\rightarrow\chi,\psi\rightarrow\chi\vdash((\varphi\rightarrow\psi)\rightarrow\psi)\rightarrow\chi$
\end{enumerate}
where $\neg\varphi:=\varphi\rightarrow0$.

Moreover, taking $\varepsilon(p)=p$, $\delta(p)=p\rightarrow p$ and $\Delta(p,q)=\{p\rightarrow q,q\rightarrow p\}$, it is known (see \cite{BP}) that every SIC has an algebraic semantics with the defining identity $\delta\approx\varepsilon$ and with the set $\Delta$ as an equivalence system. As a consequence we obtain

The logic $(\mathcal L,\vdash_{{\rm OMP}})$ is algebraizable with equivalence formulas $\Delta=\{p\rightarrow q,q\rightarrow p\}$ and the defining identity $p\approx p\rightarrow p$. This will be proved in details below.

In order to show that the system is really an axiom system for orthomodular posets in Gentzen style, we prove the following important properties.

\begin{lemma}\label{lem1}
In the propositional logic $L_{{\rm OMP}}$ the following are provable:
\begin{enumerate}[{\rm(i)}]
\item $\varphi\rightarrow\psi,\psi\rightarrow\chi\vdash\varphi\rightarrow\chi$,
\item $\vdash\varphi\rightarrow(0\rightarrow0)$.
\end{enumerate}
\end{lemma}

\begin{proof}
\
\begin{enumerate}[(i)]
\item We have
\[
\varphi\rightarrow\psi\vdash(\psi\rightarrow\chi)\rightarrow(\varphi\rightarrow\chi)
\]
by (Sf) and
\[
\psi\rightarrow\chi,(\psi\rightarrow\chi)\rightarrow(\varphi\rightarrow\chi)\vdash\varphi\rightarrow\chi
\]
by (MP).
\item We have
\[
\vdash0\rightarrow0,
\]
by (B2) or (B4),
\[
\vdash(0\rightarrow0)\rightarrow(\varphi\rightarrow(0\rightarrow0))
\]
by (B1) and
\[
0\rightarrow0,(0\rightarrow0)\rightarrow(\varphi\rightarrow(0\rightarrow0))\vdash\varphi\rightarrow(0\rightarrow0)
\]
by (MP).
\end{enumerate}
\end{proof}

In our terminology, $0\rightarrow0$ is $\neg0$ and hence it will be considered as an algebraic constant $1$. Then $0$ has the meaning of the logical value FALSE and $1$ the meaning of its opposite, i.e.\ the logical value TRUE.

In order to show that our system is an equivalent algebraic semantics for $(\mathcal L,\vdash_{{\rm OMP}})$ we use the following statement (see \cite{BP}, Theorem~2.17).

\begin{proposition}
Let $(\mathcal L,\vdash_L)$ be a deductive system given by a set of axioms {\rm Ax} and a set of inference rules {\rm Ir}. Assume $(\mathcal L,\vdash_L)$ is algebraizable with equivalence formulas $\Delta$ and defining identities $\delta\approx\varepsilon$. Then the unique equivalent semantics for $(\mathcal L,\vdash_L)$ is axiomatized by the identities
\begin{itemize}
\item $\delta(\varphi)\approx\varepsilon(\varphi)$ for each $\varphi\in{\rm Ax}$,
\item $\delta(p\Delta p)\approx\varepsilon(p\Delta p)$
\end{itemize}
together with the quasiidentities
\begin{itemize}
\item $\delta(\psi_0)\approx\varepsilon(\psi_0)\wedge\cdots\wedge\delta(\psi_{n-1})\approx\varepsilon(\psi_{n-1})\Rightarrow\delta(\varphi)\approx\varepsilon(\varphi)$ for each $\psi_0,\ldots,\psi_{n-1}\vdash\varphi\in{\rm Ir}$,
\item $\delta(p\Delta q)\approx\varepsilon(p\Delta q)\Rightarrow p\approx q$.
\end{itemize}
\end{proposition}

At first, we prove that the logic $L_{{\rm OMP}}$ is axiomatizable in the sense of \cite{BP}. 

It is evident that our system $(\mathcal L,\vdash_{{\rm OMP}})$ satisfies the aforementioned properties and hence it is a SIC.

\begin{theorem}\label{th3}
The logic $L_{{\rm OMP}}$ is algebraizable.
\end{theorem}

\begin{proof}
By Theorem~4.7 of \cite{BP}, it suffices to prove the following statements for all formulas $\varphi,\psi,\chi$ in $L_{{\rm OMP}}$:
\begin{enumerate}
\item[(i)] $\vdash\varphi\rightarrow\varphi$,
\item[(ii)] $\varphi\rightarrow\psi\vdash\varphi\rightarrow\psi$,
\item[(iii)] $\varphi\rightarrow\psi,\psi\rightarrow\chi\vdash\varphi\rightarrow\chi$,
\item[(iv-1)] $\varphi\rightarrow\psi\vdash(\psi\rightarrow\chi)\rightarrow(\varphi\rightarrow\chi)$,
\item[(iv-2)] $\varphi\rightarrow\psi,\psi\rightarrow\varphi\vdash(\chi\rightarrow\varphi)\rightarrow(\chi\rightarrow\psi)$,
\item[(v-1)] $\varphi\vdash\varphi\rightarrow(0\rightarrow0)$,
\item[(v-2)] $\varphi\vdash(0\rightarrow0)\rightarrow\varphi$,
\item[(v-3)] $(0\rightarrow0)\rightarrow\varphi\vdash\varphi$.
\end{enumerate}
Now we prove these conditions.
\begin{enumerate}
\item[(i)] This is just (B2).
\item[(ii)] We have
\[
\vdash(\varphi\rightarrow\psi)\rightarrow(\varphi\rightarrow\psi)
\]
by (B2) and
\[
\varphi\rightarrow\psi,(\varphi\rightarrow\psi)\rightarrow(\varphi\rightarrow\psi)\vdash\varphi\rightarrow\psi
\]
by (MP).
\item[(iii)] This is just (i) of Lemma~\ref{lem1}.
\item[(iv-1)] This is just (Sf).
\item[(iv-2)] We have
\[
\varphi\rightarrow\psi,\psi\rightarrow\varphi\vdash\varphi\approx\psi
\]
by (R1) and
\[
\vdash(\chi\rightarrow\varphi)\rightarrow(\chi\rightarrow\varphi)
\]
by (B2).
\item[(v-1)] This follows from (ii) of Lemma~\ref{lem1}.
\item[(v-2)] We have
\[
\vdash\varphi\rightarrow((0\rightarrow0)\rightarrow\varphi)
\]
by (B1) and
\[
\varphi,\varphi\rightarrow((0\rightarrow0)\rightarrow\varphi)\vdash(0\rightarrow0)\rightarrow\varphi
\]
by (MP).
\item[(v-3)] We have
\[
\vdash0\rightarrow0
\]
by (B2) or (B4) and
\[
0\rightarrow0,(0\rightarrow 0)\rightarrow\varphi\vdash\varphi
\]
by (MP).
\end{enumerate}
\end{proof}

Now we show that our Gentzen system is in fact valid in the class of implication orthomodular posets. For this purpose, we only compare the given axioms and rules of $L_{{\rm OMP}}$ with the axioms (O1) -- (O9) of Definition~\ref{def2}.

\begin{theorem}
Axioms {\rm(B1)} -- {\rm(B5)} and rules {\rm(MP)}, {\rm(Sf)} and {\rm(R1)} -- {\rm(R5)} are valid in every implication orthomodular poset.
\end{theorem}

\begin{proof}
\
\begin{enumerate}
\item[(B1)] follows from (O10).
\item[(B2)] follows from (O1).
\item[(B3)] follows from (O4).
\item[(B4)] follows from (O1).
\item[(B5)] follows from (O9).
\item[(MP)] We have $x\rightarrow1\approx x\rightarrow(0\rightarrow x)\approx1$ by (O1) and (O10). Hence, if $x=x\rightarrow y=1$ then $y=(((y'\rightarrow1)\rightarrow1)'\rightarrow1)\rightarrow1=1$ by (O6).
\item[(Sf)] follows from (O5).
\item[(R1)] follows from (O2).
\item[(R2)] follows from (O6).
\item[(R3)] and (R4) follow from (O7).
\item[(R5)] follows from (O8).
\end{enumerate}
\end{proof}

Finally, we prove that also conversely, the axiom system $L_{{\rm OMP}}$ in fact induces the class of implication orthomodular posets. This shows that this system is a proper axiom system in Gentzen style for orthomodular posets of finite height.

\begin{theorem}
Using axioms {\rm(B1)} -- {\rm(B5)} and rules {\rm(MP)}, {\rm(Sf)} and {\rm(R1)} -- {\rm(R5)}, we can derive {\rm(O1)} -- {\rm(O10)}.
\end{theorem}

\begin{proof}
\
\begin{enumerate}[(O1)]
\item follows from (B2) and (B4).
\item follows from (R1).
\item follows from (iii) of the proof of Theorem~\ref{th3}.
\item follows from (B3).
\item follows from (Sf).
\item follows from (R2).
\item follows from (R3) and (R4).
\item follows from (R5).
\item follows from (B5).
\item follows from (B1).
\end{enumerate}
\end{proof}

\section{Conclusion}

We have shown that there can be introduced a connective implication in a natural way also in an orthomodular poset $\mathbf P$ under the condition that $\mathbf P$ is of a finite height, i.e.\ if every chain in $\mathbf P$ is finite. Such an implication can be characterized by ten simple axioms. The resulting algebra called an implication orthomodular poset can be easily constructed from the original orthomodular poset $\mathbf P$ and, conversely, this orthomodular poset $\mathbf P$ can be obtained back from its assigned implication orthomodular poset, in other words, the last can be considered as an algebraic representation of $\mathbf P$. We derive a logical system of five simple axioms and seven derivation rules in Gentzen style which characterize implication orthomodular posets and which is algebraizable in the sense of Blok and Pigozzi. Hence, it justifies to call this deductive system a logic of orthomodular posets.

Authors' addresses:

Ivan Chajda \\
Palack\'y University Olomouc \\
Faculty of Science \\
Department of Algebra and Geometry \\
17.\ listopadu 12 \\
771 46 Olomouc \\
Czech Republic \\
ivan.chajda@upol.cz

Helmut L\"anger \\
TU Wien \\
Faculty of Mathematics and Geoinformation \\
Institute of Discrete Mathematics and Geometry \\
Wiedner Hauptstra\ss e 8-10 \\
1040 Vienna \\
Austria, and \\
Palack\'y University Olomouc \\
Faculty of Science \\
Department of Algebra and Geometry \\
17.\ listopadu 12 \\
771 46 Olomouc \\
Czech Republic \\
helmut.laenger@tuwien.ac.at

\begin{thebibliography}{99}
\bibitem{Be}
L.~Beran, Orthomodular Lattices. Algebraic Approach. Reidel, Dordrecht 1985. ISBN 90-277-1715-X.
\bibitem{Bi}
G.~Birkhoff, Lattice Theory. AMS, Providence, R.~I., 1979. ISBN 0-8218-1025-1.
\bibitem{BP}
W.~J.~Blok and Don~Pigozzi, Algebraizable logics. Mem.\ Amer.\ Math.\ Soc.\ {\bf77} (1989).
\bibitem{CC}
I.~Chajda and J.~Cirulis, An implicational logic for orthomodular lattices. Acta Sci.\ Math.\ (Szeged) {\bf82} (2016), 383--394.
\bibitem{CHL}
I.~Chajda, R.~Hala\v s and H.~L\"anger,	The logic induced by effect algebras.	Soft Computing (submitted).	http://arxiv.org/abs/2001.06686.
\bibitem{CLa}
I.~Chajda and H.~L\"anger, Implication in weakly and dually weakly orthomodular lattices.	Proc.\ 2018 Conf.\ Algebra Substructural Logics Take 6 (to appear).
\bibitem{CLb}
I.~Chajda and H.~L\"anger, How to introduce the connective implication in orthomodular posets. Asian-European J. Math.\ (submitted). http://arxiv.org/abs/1907.10539.
\bibitem K
G.~Kalmbach, Orthomodular Lattices, Academic Press, London 1983. ISBN 0-12-394580-1.
\bibitem R
H.~Rasiowa, An algebraic approach to non-classical logics. North-Holland, Amsterdam 1974.
\bibitem{RS}
H.~Rasiowa and R.~Sikorski, The mathematics of metamathematics. PWN, Warsaw 1970.
\end{thebibliography}
\end{document}